\documentclass[
	english
]{scrartcl}

\usepackage[T1]{fontenc}
\usepackage[utf8]{inputenc}

\usepackage{amsmath,amsfonts,amsthm,amssymb}
\usepackage[colorlinks,linkcolor=blue,citecolor=blue]{hyperref}

\usepackage{preamble}
\usepackage{booktabs,multirow}

\usepackage{subfig}

\usepackage{enumitem}
\setenumerate{label=(\roman*)}

\numberwithin{equation}{section} 

\usepackage[capitalise,nameinlink]{cleveref}

\usepackage{lmodern}
\usepackage[final]{microtype}

\usepackage[normalem]{ulem}

\usepackage{mathrsfs}

\usepackage{algpseudocode}

\newif\ifbiber
\bibertrue

\ifbiber
\usepackage[
	style=authoryear-comp,
	sorting=nyt,
	url=true,
	doi=true,                
	isbn=false,              
	eprint=true,
	giveninits=true,         
	uniquename=init,         
	maxbibnames=99,          
	maxcitenames=3,
	uniquelist=minyear,
	dashed=false,            
	sortcites=false,          
	backend=biber,           
	safeinputenc,            
	date=year,               
]{biblatex}

\DeclareCiteCommand{\cite}{%
	\ifbibmacroundef{cite:init}{}{\usebibmacro{cite:init}}\usebibmacro{prenote}%
}{%
	\usebibmacro{citeindex}%
	\printtext[bibhyperref]{\usebibmacro{cite}}%
}{%
	\ifbibmacroundef{cite:init}{\multicitedelim}{}%
}{%
	\usebibmacro{postnote}%
}%
\DeclareCiteCommand{\parencite}[\mkbibbrackets]{%
	\ifbibmacroundef{cite:init}{}{\usebibmacro{cite:init}}\usebibmacro{prenote}%
}{%
	\usebibmacro{citeindex}%
	\printtext[bibhyperref]{\usebibmacro{cite}}%
}{%
	\ifbibmacroundef{cite:init}{\multicitedelim}{}%
}{%
	\usebibmacro{postnote}%
}%

\let\cite\parencite

\DeclareNameAlias{sortname}{family-given}
\DeclareNameAlias{default}{family-given}

\else

\usepackage{doi,natbib}

\fi

\usepackage{mathtools}

\DeclarePairedDelimiterX\set[1]\{\}{#1}
\DeclarePairedDelimiterX\seq[1](){#1}

\DeclarePairedDelimiterX\dual[2]{\langle}{\rangle}{#1,#2}
\DeclarePairedDelimiterX\innerprod[2](){#1,#2}

\DeclarePairedDelimiter\norm{\lVert}{\rVert}

\DeclarePairedDelimiter\bracks[]

\newcommand\N{\mathbb{N}}
\newcommand\R{\mathbb{R}}

\renewcommand\d{\mathop{}\!\mathrm{d}}

\newcommand\dt{\d t}

\newcommand{\dualspace}{^*}

\DeclareMathOperator\id{id}

\DeclareMathOperator\lin{lin}

\DeclareMathAlphabet{\mathpzc}{OT1}{pzc}{m}{it}

\newcommand\orcid[1]{%
	\hspace{.25em}%
	\href{http://orcid.org/#1}{%
		\protect\includegraphics[height=1em]{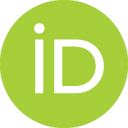}%
	}%
	\hspace{.25em}%
}

\newtheorem{theorem}{Theorem}[section]
\newtheorem{lemma}[theorem]{Lemma}
\newtheorem{assumption}[theorem]{Assumption}

\newtheorem{definition}[theorem]{Definition}

\crefname{assumption}{Assumption}{Assumptions}

\def\cleverreffix#1{\AddToHook{env/#1/begin}{\crefalias{theorem}{#1}}}
\cleverreffix{lemma}
\cleverreffix{assumption}
\cleverreffix{corollary}
\cleverreffix{remark}
\cleverreffix{definition}
\cleverreffix{example}

\begin{document}
\title{
	Characterization of Hilbertizable spaces via convex functions
}

\author{%
	Nicolas Borchard%
	\footnote{%
		Brandenburgische Technische Universität Cottbus-Senftenberg,
		Institute of Mathematics,
		03046 Cottbus, Germany,
		\email{nicolas.borchard@b-tu.de}%
	}
	\orcid{0009-0007-7358-7737}%
	\and
	Gerd Wachsmuth%
	\footnote{%
		Brandenburgische Technische Universität Cottbus-Senftenberg,
		Institute of Mathematics,
		03046 Cott\-bus, Germany,
		\email{gerd.wachsmuth@b-tu.de}%
	}
	\orcid{0000-0002-3098-1503}%
}

\maketitle

\begin{abstract}
	We show that the existence of a strongly convex function with a Lipschitz derivative
	on a Banach space
	already implies that the space is isomorphic to a Hilbert space.
	Similarly, if both a function and its convex conjugate are $C^2$ then the underlying space is also isomorphic to a Hilbert space.
\end{abstract}


\begin{msc}
	46C15,
	46N10
\end{msc}

\section{Introduction}

Convex functions which are strongly convex and which have a Lipschitz continuous derivative
are very useful
for, e.g., analyzing the convergence of first-order optimization methods.
In particular, it allows for an estimate of the convergence rate of
the gradient method, see \cite[Theorem~2.1.15]{Nesterov2004}.
This theorem is posed in $\R^n$,
but can be generalized to Hilbert spaces.
One could pose the question
whether it can be further generalized to arbitrary Banach spaces.
However, we will show that
the existence of a strongly convex function with Lipschitz continuous derivative
already implies that the space is isomorphic to a Hilbert space,
see \cref{thm:isomorphic_to_Hspace}.
Conversely,
given a space which is not isomorphic to a Hilbert space
(like $\ell^p$ or $L^p(0,1)$ for $p \ne 2$),
it is not possible to construct a strongly convex function
which has a Lipschitz continuous derivative.

There have been other characterizations of Hilbertizable spaces,
i.e., spaces which are isomorphic to a Hilbert space.
In the context of infinite-dimensional optimization,
it was shown in \cite{Harder2018}
that if a Legendre form exists on a reflexive Banach space, then the space is
isomorphic to a Hilbert space.

Properties of uniformly convex and uniformly smooth convex functions
(which generalize strongly convex functions and convex functions with Lipschitz derivatives)
can be found in
\cite{AzePenot1995} and \cite[Section~3.5]{Zalinescu2002}.
In particular,
these results can be applied to strongly convex functions
and convex functions with Lipschitz derivatives.

We extend \cref{thm:isomorphic_to_Hspace}
to monotone operators,
that is, the existence of a strongly monotone and Lipschitz continuous operator
implies that the underlying space is Hilbertizable,
see \cref{thm:operator}.
Note that the existence of such an operator
on a Banach space
has been assumed in some works in the literature,
see, e.g.,
\cite[Theorems~4.13, 4.14 under conditions (a)*, (b)]{IzuchukwuReichShehu2022}
and
\cite[Theorem~3.1]{MendyMendy2023}.
Hence, our findings show that
these results are only applicable in Hilbertizable spaces.

Finally,
in \cref{thm:is_Hilbert}
we show
that the existence of a convex function $f$
such that $f$ and its convex conjugate $f^*$
are both $C^2$,
implies Hilbertizablility.

\section{Strong convexity and Lipschitz derivative}
In the next assumption,
we fix the class of functions
which are locally strongly convex and have a Lipschitz derivative.
\begin{assumption}
	\label{assu:C1_strongly_cvx_Lip_der}
	Let $X$ be a Banach space,
	$f : X \to \R$ be a function,
	$\bar x \in X$ and
	$\varepsilon > 0$.
	On the open ball in $X$ centered at $\bar x$ with a radius of $\varepsilon$, i.e. $U_\varepsilon^X(\bar x)$,
	let $f$ be of class $C^1$, $\mu$-strongly convex with a Lipschitz derivative for a constant $L > 0$,
	i.e.,
	\begin{align*}
		f( \lambda x + (1-\lambda) y)
		&
		\le
		\lambda f(x) + (1-\lambda) f(y) - \frac{\mu}{2} \lambda (1-\lambda) \norm{x - y}_X^2
		,
		\\
		\norm{f'(x) - f'(y)}_{X\dualspace}
		&\le
		L \norm{x - y}_X
	\end{align*}
	holds for all $x,y \in U_\varepsilon^X(\bar x)$
	and all $\lambda \in [0,1]$.
\end{assumption}
It is well known that the strong convexity of $f$
implies
\begin{equation}
	\label{eq:strong_cvx}
	\mu \norm{x - y}^2
	\le
	\dual{f'(x) - f'(y)}{ x - y }_X
	\qquad\forall x,y \in U_\varepsilon^X(\bar x)
	.
\end{equation}
Indeed,
from the strong convexity we get
\begin{equation*}
	\frac{f( \lambda x + (1-\lambda) y) - f(y)}{\lambda}
	\le
	f(x) - f(y) - \frac{\mu}{2} (1-\lambda) \norm{x - y}_X^2
\end{equation*}
and $\lambda \searrow 0$ implies
\begin{equation*}
	\dual{f'(y)}{ x - y }_X \le f(x) - f(y) - \frac\mu2 \norm{x - y}^2
	.
\end{equation*}
The same inequality holds with $x$ and $y$ interchanged
and addition of both inequalities gives
\eqref{eq:strong_cvx}.

First, we check on every finite-dimensional subspace of $X$,
we can define an equivalent norm
and the constants of equivalence are bounded by $\sqrt\mu$ and $\sqrt L$.
\begin{lemma}
	\label{lem:inner_product_on_fin_dim_space}
	Let \Cref{assu:C1_strongly_cvx_Lip_der} be fulfilled.
	For every finite-dimensional subspace $U$ of $X$,
	there exists an inner product $\innerprod{\cdot}{\cdot}_U : U \times U \to \R$
	such that the norm $\norm{\cdot}_X$ and $\norm{\cdot}_U$ are equivalent.
	In particular,
	\[
		\sqrt{\mu} \norm{x}_X
		\le
		\norm{x}_U
		\le
		\sqrt{L} \norm{x}_X
		\qquad\forall x \in U
	\]
	and all finite-dimensional subspaces $U$.
\end{lemma}
\begin{proof}
	Without loss of generality, let $\bar x = 0$.
	We define the function $g := f|_U : U \to \R$.
	Obviously, $g$ is strongly convex and has a Lipschitz derivative
	in a neighborhood of $0$ and it inherits the constants of $f$.
	As $g'$ is a locally Lipschitz continuous function between finite-dimensional spaces,
	we can employ the well-known Rademacher theorem and obtain $\bar u \in U$ such that $g'$ is differentiable in $\bar u$.
	We set $A := g''(\bar u) : U \times U \to \R$.
	Note that $A$ is symmetric due to \cite[Theorem~5.1.1]{Cartan1967}.
	Since $g$ is twice differentiable at $\bar u$,
	we have
	\begin{equation*}
		\frac{g'(\bar u + th) - g'(\bar u)}{t}
		\to
		A[h,\cdot]
		\qquad\text{as } t \to 0
	\end{equation*}
	for each fixed $h \in U$.
	From the Lipschitz continuity of $f'$ and
	\eqref{eq:strong_cvx}
	we get
	\begin{equation*}
		\mu \norm{h}_X^2
		\le
		\dual*{\frac{g'(\bar u + th) - g'(\bar u)}{t}}{ h }_U
		=
		\dual*{\frac{f'(\bar u + th) - f'(\bar u)}{t}}{ h }_X
		\le
		L \norm{h}_X^2
		.
	\end{equation*}
	Consequently,
	\begin{equation*}
		\mu \norm{h}_X^2
		\le
		A[h,h]
		\le
		L \norm{h}_X^2
		\qquad\forall h \in U
		.
	\end{equation*}
	In total, $A$ is an inner product
	as it is bilinear, symmetric, and positive-definite.
	The norm generated by $A$ is equivalent to the original norm
	with constants $\sqrt{\mu}$ and $\sqrt{L}$.
\end{proof}

\begin{definition}
	\cite[Section~9.2]{LedouxTalagrand1991}
	\label{def:type_cotype}
	For every $n \in \N$, we define $E(n) := \set{-1,1}^n$.
	A Banach space $X$ is said to be of type $q \in [1,2]$ if
	there exists a constant $C > 0$
	such that
	\[
		\frac 1 {2^n} \sum_{\varepsilon \in E(n)} \norm*{\sum_{k=1}^n \varepsilon_k x_k}_X^q
		\le
		C^q \sum_{k=1}^n \norm{x_k}_X^q
	\]
	holds for all $n \in \N$ and all $x_1, \ldots, x_n \in X$.
	The infimum over all possible constants is denoted by $T_q(X)$.
	Similarly, it is of cotype $q \in [2, \infty)$ if
	there exists a constant $C > 0$ such that
	\[
		\frac 1 {C^q} \sum_{k=1}^n \norm{x_k}_X^q
		\le
		\frac 1 {2^n} \sum_{\varepsilon \in E(n)} \norm*{\sum_{k=1}^n \varepsilon_k x_k}_X^q
	\]
	holds for all $n \in \N$ and all $x_1, \ldots, x_n \in X$.
	The infimum over all possible constants is denoted by $C_q(X)$.
\end{definition}

Let $(\Omega, \mathcal{A}, \mu)$ be a measure space.
Is is known from \cite[p.~247]{LedouxTalagrand1991} that
$L^p(\mu)$ for $p \in [1,2]$
is of type $q$ for all $q \in [p,2]$
and of cotype $q$ for all $q \in \set{2}$.
Similarly,
$L^p(\mu)$ for $p \in [2,\infty)$
is of type $q$ for all $q \in \set{2}$
and of cotype $q$ for all $q \in [p, \infty)$.
In the case that $L^p(\mu)$ is infinite-dimensional,
it is of no other type or cotype.
This covers the special cases $L^p(\Omega)$ (for $\Omega \subset \R^d$ equipped with the Lebesgue measure)
and $\ell^p$.

An easy calculation shows that
any Hilbert space $H$ has type $2$ and cotype $2$,
and we have $C_2(H) = T_2(H) = 1$
(unless $H = \set{0}$).
It is easy to check that the type and cotype of a Banach space
does not change if one uses an equivalent norm.
Hence,
any space which is isomorphic to a Hilbert space
also has type $2$ and cotype $2$.
In particular,
the spaces $\ell^p$ are not isomorphic to a Hilbert space
for $p \ne 2$.
The same applies to $L^p(\mu)$ if it is infinite dimensional.

Moreover, the next result shows that the property of
having type $2$ and cotype $2$
already
characterizes Hilbertizable spaces.
\begin{theorem}
	\cite[Prop.~3.1]{Kwapien1972}
	\label{thm:kwapien}
	A space is isomorphic to a Hilbert space if and only if it is of type $2$ and cotype $2$.
\end{theorem}
Note that in \cite{Kwapien1972}, the names type and cotype were not yet established.
Kwapień used the notion of the so called two-sided Khinchine inequality.
We also refer to \cite{Yamasaki1984} for a simple proof of Kwapień's theorem.
From this proof, one obtains that
for every space $X$ with type $2$ and cotype $2$,
there exists a Hilbert space
such that
the Banach--Mazur distance
between $X$ and $H$
is bounded by the product $T_2(X) C_2(X)$.

Now, we are in position to prove our main result.
\begin{theorem}
	\label{thm:isomorphic_to_Hspace}
	Let \Cref{assu:C1_strongly_cvx_Lip_der} be fulfilled.
	Then, $X$ is isomorphic to a Hilbert space.
\end{theorem}
\begin{proof}
	We use \Cref{thm:kwapien}.
	Let $n \in \N$ and $x_1, \ldots, x_n \in X$
	be arbitrary.
	We set $U := \lin(x_1, \ldots, x_n)$.
	We show the second inequality first.
	As $U$ is finite-dimensional,
	\Cref{lem:inner_product_on_fin_dim_space} can be applied
	and yields a scalar product $\innerprod{\cdot}{\cdot}_U$ on $U$.
	Thus,
	\begin{align*}
		\frac \mu {2^n} \sum_{\varepsilon \in E(n)} \norm*{\sum_{k=1}^n \varepsilon_k x_k}_X^2
		&\le
		\frac {1} {2^n}
		\sum_{\varepsilon \in E(n)} \norm*{\sum_{k=1}^n \varepsilon_k x_k}_U^2
		\\
		&=
		\frac {1} {2^n}
		\sum_{\varepsilon \in E(n)}
		\left(
		\sum_{k=1}^n \norm{x_k}_U^2
		+
		\sum_{\substack{k,l=1 \\ k \ne l}}^n \varepsilon_k \varepsilon_l \innerprod{x_k}{x_l}_U
		\right)
	\end{align*}
	The first sum inside the brackets does not depend on $\varepsilon$ anymore, thus the same sum is added $2^n$ times.
	The second inner sum can be swapped with the outer sum over $\varepsilon$
	and, due to the definition of $E(n)$, we have
	$\sum_{\varepsilon \in E_n} \varepsilon_k \varepsilon_l = 0$ for $k \ne l$.
	Consequently, the second inner sum disappears.
	Thus we get
	\begin{equation*}
		\frac \mu {2^n} \sum_{\varepsilon \in E(n)} \norm*{\sum_{k=1}^n \varepsilon_k x_k}_X^2
		\le
		\left(
		\sum_{k=1}^n \norm{x_k}_U^2
		+
		\frac 1 {2^{n}}
		\sum_{\substack{k,l=1 \\ k \ne l}}^n
		(2^{n-1} - 2^{n-1}) \innerprod{x_k}{x_l}_U
		\right)
		\le
		L \sum_{k=1}^n \norm{x_k}_X^2
		.
	\end{equation*}
	Since the constants $\mu$ and $L$
	do not depend on the subspace $U$,
	this shows that $X$ is of type $2$.
	Analogously, one can reverse all these inequalities while swapping $\mu$ and $L$.
	Thus $X$ is of cotype $2$.
	\cref{thm:kwapien} yields the claim.
\end{proof}

From the proof, we see that the Banach--Mazur distance of $X$
to a Hilbert space can be bounded by $\sqrt{L/\mu}$.

Finally,
we check that these considerations
can be slightly generalized.
\begin{theorem}
	\label{thm:operator}
	Let $X$ be a Banach space.
	Further, let an operator $T \colon X \to X\dualspace$ be given
	such that $T$ is Lipschitz continuous and strongly monotone
	on $U_\varepsilon^X(\bar x)$ for some $\bar x \in X$ and $\varepsilon > 0$,
	i.e.,
	there exist $\mu, L > 0$
	such that
	\begin{align*}
		\norm{T(x) - T(y)}_{X\dualspace} &\le L \norm{x - y}_X,  \\
		\dual{T(x) - T(y)}{x - y}_X &\ge \mu \norm{x - y}_X^2
	\end{align*}
	holds for all $x,y \in U_\varepsilon^X(\bar x)$.
	Then $X$ is isomorphic to a Hilbert space.
\end{theorem}
\begin{proof}
	We briefly check that the assertion of \cref{lem:inner_product_on_fin_dim_space}
	is valid in our situation.
	In fact, we can repeat the proof of \cref{lem:inner_product_on_fin_dim_space}
	with some minor modifications.
	For a finite-dimensional subspace $U$, we consider the operator
	$S \colon U \to U\dualspace$, $U \ni u \mapsto T(u)|_U \in U\dualspace$.
	Due to Rademacher's theorem, $S$ is differentiable at some $\bar u \in U$
	and we set $B := S'(\bar u) \in \mathcal{L}(U, U\dualspace)$.
	The only difference to the situation in \cref{lem:inner_product_on_fin_dim_space}
	is that the bilinear form associated with $B$ might not be symmetric,
	but we can symmetrize $A := (B + B^*)/2$.
	The remaining steps of the proof are identical.
\end{proof}

\section{Function and its conjugate being $C^2$}
In this section we consider the case that a function and its convex conjugate
are twice continuously differentiable.
\begin{theorem}
	\label{thm:is_Hilbert}
	Let $X$ be a Banach space,
	$f \in C^2(X)$ is convex
	and its convex conjugate satisfies $f^* \in C^2(X^*)$.
	Then, $X$ is isomorphic to a Hilbert space.
\end{theorem}
In the case that $X$ is additionally reflexive,
we give a simple and self-contained proof.
\begin{proof}[Proof ($X$ reflexive)]
	Let $\bar x \in X$ be arbitrary.
	We define
	$A := f''(\bar x) \in \mathcal{L}(X,X^*)$ and $B := (f^*)''(f'(\bar x)) \in \mathcal{L}(X^*,X)$.
	
	It is well-known that $x^* \in \partial f(x)$ and $x \in \partial f^*(x^*)$ are equivalent.
	This and the differentiability imply
	$x = (f^*)'(f'(x))$ for all $x \in X$
	and
	$x^* = f'((f^*)'(x^*))$ for all $x^* \in X^*$.
	We differentiate this and obtain
	$\id_X = (f^*)''(f'(\bar x)) f''(\bar x) = BA$
	and
	$\id_{X^*} = A B$.
	In particular, $A$ and $B$ are invertible.
	Note that this argument fails in case that $X$ is not reflexive.
	
	We define the function $g : X \to \R$ via
	$g(x) := \frac 1 2 A[x,x] := \frac 1 2 \dual{Ax}{x}_X$.
	The symmetry of $A[\cdot, \cdot]$ implies
	$\dual{A^* x}{y}_X = \dual{Ay}{x}_X = A[y,x] = A[x,y] = \dual{Ax}{y}_X$
	for all $x,y \in X$
	and thus self-adjointness.
	Similarly, $B$ is self-adjoint.
	We calculate the convex conjugate of $g$ by
	\begin{align*}
		g^*(x^*)
		&=
		\sup_{x \in X} \dual{x^*}{x}_X - \frac 1 2 A[x,x]
		=
		\sup_{x \in X} \frac 1 2 B[x^*,x^*] - \frac 1 2 A[Bx^* - x, Bx^* - x]
		\\&
		=
		\frac 1 2 B[x^*,x^*]
		.
	\end{align*}
	Here, we used the symmetry of $A$ and $\id_{X^*} = A B$.
	Further, the derivative of $g^*$ is given by
	$(g^*)'(x^*) = \frac 1 2 B x^* + \frac 1 2 B^* x^* = B x^* $ which is a Lipschitz continuous function.
	It is well-known that this implies strong convexity of $g$ for some constant $\mu > 0$,
	see, e.g., \cite[Proposition~3.5.3]{Zalinescu2002}.
	
	Strong convexity of $g$ and $g'(0) = A(0) = 0$ yield
	\[
	\frac 1 2 A[x,x]
	=
	g(x) - g(0) - g'(0)(x - 0)
	\ge
	\frac \mu 2 \norm{x - 0}_{X}^2
	.
	\]
	Thus, we get
	$\mu \norm{x}_{X}^2 \le A[x,x] \le \norm{A} \norm{x}_{X}^2$.
	This shows that $A[\cdot,\cdot]$ is an inner product on $X$
	and the associated norm is equivalent to the original norm.
\end{proof}
Note that the function $g$ constructed in the proof fulfills \Cref{assu:C1_strongly_cvx_Lip_der},
but we do not employ \cref{thm:isomorphic_to_Hspace}
explicitly.

In absence of the reflexivity of $X$
we show that the assumptions of \cref{thm:is_Hilbert}
imply \cref{assu:C1_strongly_cvx_Lip_der}
and this yields a proof of \cref{thm:is_Hilbert}
in the general situation.
\begin{proof}[Proof (general case)]
	W.l.o.g., we assume $0 = f'(0)$
	(otherwise we subtract the linear function $f'(0) \in X^*$ from $f$).
	As $f^*$ is $C^2$, there exist constants $\varepsilon, L > 0$ with $\norm{(f^*)''(x^*)} \le L$ for all $x^* \in B_\varepsilon^{X^*}(0)$.
	Thus, for all $x^*,y^* \in B_\varepsilon^{X^*}(0)$ it holds
	\[
	\norm{(f^*)'(x^*) - (f^*)'(y^*)}_{X^{**}}
	\le
	\sup_{\xi \in [x^*,y^*]} \norm{(f^*)''(\xi)} \norm{x^* - y^*}_{X^*}
	\le
	L \norm{x^* - y^*}_{X^*}
	,
	\]
	where we used the mean-value inequality,
	see \cite[Proposition~3.3.1]{Cartan1967}.
	Thus, $(f^*)'$ is locally Lipschitz at $0$.
	
	We are going to prove a local descent lemma.
	Let $\varphi(t) := f^*(x^* + t(y^*-x^*))$. Obviously, $\varphi$ is $C^1$.
	Thus, for $x^*,y^* \in B_\varepsilon^{X^*}(0)$
	we have
	\begin{align*}
		\MoveEqLeft
		f^*(y^*) - f^*(x^*)
		-
		\dual{(f^*)'(x^*)}{y^*-x^*}_{X^*}
		=
		\varphi(1) - \varphi(0) - \varphi'(0)
		\\
		&=
		\int_0^1 \varphi'(t) - \varphi'(0) \dt
		=
		\int_0^1 \dual{(f^*)'(x^*+t(y^*-x^*)) - (f^*)'(x^*)}{y^*-x^*}_{X^*} \dt
		.
	\end{align*}
	As $x^*+t(y^*-x^*)$ and $x^*$ are both in $B_\varepsilon^{X^*}(0)$, we can use the local Lipschitz continuity and get
	\begin{align*}
		f^*(y^*) - f^*(x^*) - \dual{(f^*)'(x^*)}{y^*-x^*}_{X^*}
		&\le
		\int_0^1 L t \norm{y^*-x^*}_{X^*}^2 \dt
		=
		\frac L 2 \norm{y^*-x^*}_{X^*}^2
		.
	\end{align*}
	As $f'$ is continuous, we find $\delta > 0$ with
	$\norm{f'(x)}_{X^*} = \norm{f'(x) - f'(0)}_{X^*} \le \varepsilon/2$
	for all $x$ with $\norm{x - 0}_{X} \le \delta$.
	
	Now, we can show local strong convexity.
	Let $\lambda \in (0,1)$ and $x,y \in B_{\min(\delta, \varepsilon L / 4)}^{X}(0)$ be given.
	Further, let $x^*,y^* \in B_{\varepsilon/2}^{X^*}(0)$ be arbitrary.
	For these we have
	$\norm{x^* - \lambda y^*}_{X^*}, \norm{x^* + (1-\lambda) y^*}_{X^*} \le \norm{x^*}_{X^*} + \norm{y^*}_{X^*} \le \varepsilon$
	which allows us to use the above local descent lemma.
	Together with the Fenchel--Young inequality, this yields
	\begin{align*}
		\MoveEqLeft
		\lambda f(x) + (1-\lambda) f(y)
		\\
		&\ge
			\lambda
			\left[
			\dual{x^*+(1-\lambda)y^*}{x} - f^*(x^*+(1-\lambda)y^*)
			\right]			
			\\
			&\qquad+ (1-\lambda)
			\left[
			\dual{x^*-\lambda y^*}{y} - f^*(x^*-\lambda y^*)
			\right]
		\\
		&\ge
			\lambda
			\left[
				\dual{x^*+(1-\lambda)y^*}{x} - f^*(x^*)
				- \dual{(f^*)'(x^*)}{(1-\lambda)y^*} - \frac L 2 \norm{(1-\lambda)y^*}_{X^*}^2
			\right]
			\\
			&\qquad+ (1-\lambda)
			\left[
				\dual{x^*-\lambda y^*}{y} - f^*(x^*)
				- \dual{(f^*)'(x^*)}{-\lambda y^*} - \frac L 2 \norm{-\lambda y^*}_{X^*}^2
			\right]
		\\
		&=
			\lambda \dual{x^*+(1-\lambda)y^*}{x}
			+ (1-\lambda) \dual{x^*-\lambda y^*}{y}
			- f^*(x^*)
			- \frac L 2 \lambda (1-\lambda) \norm{y^*}_{X^*}^2
		\\
		&=
		\bracks[\big]{
		\dual{x^*}{\lambda x + (1-\lambda) y}
		- f^*(x^*)
		}
		+ \lambda (1-\lambda)
		\left[ \dual{y^*}{x - y} - \frac L 2 \norm{y^*}_{X^*}^2 \right]
	\end{align*}
	for all $x^*,y^* \in B_{\varepsilon/2}^{X^*}(0)$.
	Due to
	$\norm{\lambda x + (1-\lambda) y} \le \delta$,
	we set
	$x^* = f'(\lambda x + (1-\lambda) y) \in B_{\varepsilon/2}^{X^*}(0)$
	and with this choice, the first bracketed term on the last line
	equals $f(\lambda x + (1-\lambda) y)$.
	
	Similarly,
	we select an arbitrary
	$y^* \in \partial (\frac 1 {2L} \norm{\cdot}_X^2)(x-y) \ne \emptyset$.
	Together with $(\frac 1 {2L} \norm{\cdot}_X^2)^* = \frac L {2} \norm{\cdot}_{X^*}^2$,
	we get
	\[
		\frac L 2 \norm{y^*}_{X^*}^2 + \frac 1 {2L} \norm{x-y}_X^2
		=
		\dual{y^*}{x-y}
		\le
		\norm{y^*}_{X^*} \norm{x-y}_X
	\]
	It is easy to check that this implies
	\[
		\norm{y^*}_{X^*}
		=
		\frac 1 L
		\norm{x-y}_X
		\le
		\frac 2 L \frac{\varepsilon L}{4}
		=
		\frac \varepsilon 2
	\]
	and, thus, we are allowed to insert this $y^*$ above.
	The second bracketed term becomes
	$\frac 1 {2L} \norm{x - y}_{X}^2$.
	Altogether, this shows
	\begin{equation*}
		\lambda f(x) + (1-\lambda) f(y)
		\ge
		f(\lambda x + (1-\lambda) y)
		+ \frac 1 {2L} \lambda (1-\lambda) \norm{x - y}_{X}^2
		,
	\end{equation*}
	i.e., strong convexity on $B_{\min(\delta, \varepsilon L / 4)}(0)$.
	In total,
	the function $f$ is strongly convex with constant $1/L$ on that ball and we can find a smaller ball where $f'$ is locally Lipschitz as $f$ is $C^2$.
\end{proof}

From these proofs,
it is also clear that
the assumptions of \cref{thm:is_Hilbert}
can be localized.
Indeed, it is sufficient that $f$ is $C^2$ in a neighborhood of some $x \in X$,
while $f^*$ is $C^2$ in a neighborhood of $f'(x)$.

\printbibliography

\end{document}